\documentclass[12pt,reqno]{amsart}

\usepackage{amsmath, amssymb}
\usepackage{hyperref}
\usepackage{enumerate}

\numberwithin{equation}{section}
 \theoremstyle{plain}
\newtheorem{theorem}{Theorem}[section]
\newtheorem{lemma}[theorem]{Lemma}
\newtheorem{corollary}[theorem]{Corollary}

\newcommand{\no}{\nonumber}

\begin{document}
\title[Andrews-Beck type congruences modulo powers of $5$]
{Andrews-Beck type congruences modulo powers of $5$}

\author{Nankun Hong}
\author{Renrong Mao}

\address{Center for Pure Mathematics, School of Mathematical Sciences, Anhui University,
	Hefei 230001, PR China}
\email{hongnk@ahu.edu.cn}

\address{Department of Mathematics, Soochow University,
SuZhou 215006, PR China}
\email{rrmao@suda.edu.cn}
\maketitle
\textbf{Abstract:}
	Let $NT(m,k,n)$ denote the total number of parts in the partitions of $n$ with rank congruent to $m$ modulo $k$.
Andrews proved Beck's conjecture on congruences for $NT(m,k,n)$ modulo $5$ and $7$. Generalizing Andrews' results, Chern obtain congruences for $NT(m,k,n)$ modulo $11$ and $13$. More recently, the second author use the theory of Hecke operators to establish congruences for such partition statistics modulo powers of primes $\ell\geq7$. In this paper, we obtain Andrews-Beck type congruences modulo powers of $5$.



\section{Introduction}

	A partition of a positive integer $n$ is a non-increasing sequence of positive integers whose sum is $n$. Let $p(n)$ denote the number of partitions of $n$. Ramanujan \cite{ramanujan1919some} conjectured that, if $\alpha\ge 1$ and $\delta_\alpha$ is the reciprocal modulo $5^\alpha$ of 24,  then 
	\begin{align}
		p(5^\alpha +\delta_\alpha) \equiv 0\pmod {5^\alpha}.\label{pnmod5}
	\end{align}
Using the modular equation of fifth order, G. N. Watson \cite{watson1938ramanujans} proved Ramanujan's conjecture. He also established that 
	\[
		p(5^\alpha n+\delta_\alpha \pm 5^{\alpha-1}) \equiv 0 \pmod {5^\alpha}
	\]
	with odd $\alpha\ge 3$.	
In \cite{HuntHirschhorn1981asimple},  M. D. Hirschhorn and D. C. Hunt provided a more elementary proof of \eqref{pnmod5}. Ramanujan's conjecture inspired numerous studies on congruences for many other types of partition statistics.
In this paper, we obtain Andrews-Beck type congruences modulo powers of $5$.
	
		Let $NT(m,k,n)$ denote the total number of parts in the partitions of $n$ with rank congruent to $m$ modulo $k$.
		Andrews \cite{andrews} proved the following congruences (first conjectured by Beck) 
		\begin{align} \label{ntmod5}
			&NT(1, 5, 5n+i) - NT(4,5,5n+i) \nonumber\\&+ 2NT(2,5,5n+i) - 2NT(3,5,5n+i) \nonumber\\&\equiv 0 \pmod{5},
			\intertext{ for $i=1$ or $4$ and for $j=1$ or $5$,}
			&N T(1,7,7 n+i)-N T(6,7,7 n+i)\nonumber \\
			&+N T(2,7,7 n+i)-N T(5,7,7 n+i)\nonumber \\
			&-N T(3,7,7 n+i)+N T(4,7,7 n+i) \nonumber\\
			&\equiv 0 \pmod{7}.
		\end{align}
	Andrews' work motivated studies on variations of Andrews-Beck partition statistics. See for example, \cite{Chan,Chern,lin}. 
More results on Andrews-Beck type congruences can be found in \cite{chern2,kim,yao}. 
In particular, Chern \cite{chern2} obtained congruences for $NT(m,k,n)$ modulo $11$ and $13$. 
By \cite[Corollary 1.2]{mao2022total}, we have
	\begin{align}\label{gencn}
		\sum_{n=0}^\infty c(n)q^n
		& :=\sum_{n=0}^\infty \left(N T(4,5,5 n+4)-N T(1,5,5 n+4)\right)q^n=	\frac{E_5^4}{E_1},
	\end{align}
where $
E_j:=E_j(q):=\prod_{n = 1}^\infty\left(1-q^{jn}\right).$
Applying \eqref{gencn} and the theory of Hecke operators, the second author \cite{maocon} established Andrews-Beck type congruences modulo powers of primes $\ell\geq7.$ For example, one can deduce from \cite[Theorem 1.1]{maocon} that, for primes $\ell \equiv4\pmod{5}$ and for integers $N\geq0$, we have
	\begin{align}
	&NT\left(4,5,5\left(\ell^{2m+1}N+\frac{\ell^{2m}(24\ell-1)-95}{120}\right)+4\right)\no\\&-NT\left(1,5,5\left(\ell^{2m+1}N+\frac{\ell^{2m}(24\ell-1)-95}{120}\right)+4\right)\no
	\\&\equiv 0\pmod{\ell^m}. \no
\end{align}
	The first objective of this paper is to prove the following Andrews-Beck type congruences modulo powers of $5$.
	
\begin{theorem}\label{th1}
For positive integers $\alpha$, define
	\begin{align*}
		\tilde{	\delta}_\alpha:=\left\{\begin{array}{l}
			\frac{23\cdot 5^{\alpha}-19}{24},\quad\textrm{ if $\alpha$ is odd,} \\
			\\
			\frac{19\cdot(5^{\alpha}-1)}{24},\qquad\textrm{ if $\alpha$ is even.}
		\end{array}\right.
	\end{align*}
	Then we have
	\begin{align}
		&NT\left(1,5,5^{\alpha+1}n+5\tilde{	\delta}_\alpha+4\right)\equiv NT\left(4,5,5^{\alpha+1}n+5\tilde{	\delta}_\alpha+4\right)\pmod{5^\alpha}.\label{thcon}
	\end{align}
\end{theorem}
The proof of Theorem \ref{th1} follows from a completely similar argument of Hirschorn and Hunt\cite{HuntHirschhorn1981asimple}.
The key is to establish the following identity for $c(n)$ defined in \eqref{gencn}.
\begin{theorem}\label{thid}
	For $\alpha\geq 1$, we have
	\begin{align}
		\sum_{n=0}^\infty c\left(5^\alpha n+	\tilde{	\delta}_\alpha\right)q^n=		&\sum_{i \geq 1} \tilde{x}_{\alpha, i} q^{i-1} \frac{E_5^{6 i-1}}{E_1^{6 i-4}},\qquad \textrm{ $\alpha$ odd,} \label{thidodd}
		\\	
		\sum_{n=0}^\infty c\left(5^\alpha n+	\tilde{	\delta}_\alpha\right)q^n=		&	\sum_{i \geq 1}\tilde{x}_{\alpha, i} q^{i-1} \frac{E_5^{6 i-2}}{E_1^{6 i-5}}, \qquad\textrm{ $\alpha$ even,}\label{thideven}
	\end{align}
	where
	\begin{align}
		%	E_1=\prod_{n \geq 1}\left(1-q^{n}\right) \\
		\tilde{\bf {x}}_{1}=\left(\tilde{x}_{1,1}, \tilde{x}_{1,2}, \ldots\right)=(5,0,0,0, \ldots),\no
	\end{align}
	and for $\alpha\geq 1$,
	\begin{equation}\label{def_x}
		\tilde{\bf {x}}_{\alpha+1}=\left\{\begin{array}{lll}
			\tilde{\bf {x}}_{\alpha} \tilde{A}, & \alpha & \text {odd} , 
			\\
			\\
			\tilde{\bf {x}}_{\alpha} \tilde{B}, & \alpha & \text {even} .
		\end{array}\right.
	\end{equation}
	Here $\tilde{A}=\left(\tilde{a}_{i, j}\right)_{i, j \geq 1}$ and $\tilde{B}=\left(\tilde{b}_{i, j}\right)_{i, j \geq 1}$ are defined by
	
	\begin{align}
		\tilde{a}_{i, j}=m_{6 i-4, i-1+j}, \quad \tilde{b}_{i, j}=m_{6 i-5, i-1+j},\label{defab}
	\end{align}
	where $M=\left(m_{i, j}\right)_{i, j \geq 1}$ is defined as follows:
	
	The first five rows of $M$ are
	\begin{align}\label{defm1}
		\left(\begin{array}{ccccccc}
			5 & 0 & 0 & 0 & 0 & 0 & \cdots \\
			2\cdot 5 & 5^{3} & 0 & 0 & 0 & 0 & \cdots \\
			9 & 3 \cdot 5^{3} & 5^{5} & 0 & 0 & 0 & \cdots \\
			4 & 22 \cdot 5^{2} & 4 \cdot 5^{5} & 5^{7} & 0 & 0 & \cdots \\
			1 & 4 \cdot 5^{3} & 8 \cdot 5^{5} & 5^{8} & 5^{9} & 0 & \cdots
		\end{array}\right)
	\end{align}
	and for $i \geq 6, m_{i, 1}=0$, and for $j \geq 2,$
	\begin{align}
		m_{i, j}=25 m_{i-1, j-1}+25 m_{i-2, j-1}+15 m_{i-3, j-1}+5 m_{i-4, j-1}+m_{i-5, j-1}. \label{defm2}
	\end{align}
\end{theorem}
Using Theorem \ref{thid} and a result due to Lovejoy and Ono \cite{love}, we obtain the following congruences for $c(n)$.
\begin{theorem}\label{th2}
	Let $\ell\geq7$ be prime and define
	the rational numbers $y_{\alpha,\ell}$ by
	\begin{align*}
		y_{\alpha,\ell}:=\left\{\begin{array}{l}
			\frac{23\cdot 5^{\alpha} \cdot \ell^{2}-19}{24}, \quad\textrm{ if $\alpha$ is odd,} \\
			\\
			\frac{19 \cdot 5^{\alpha} \cdot \ell^{2}-19}{24}, \quad\textrm{ if $\alpha$ is even.}
		\end{array}\right.
	\end{align*}
	\begin{enumerate}
		\item 
		If $\alpha\geq1$ is odd, then for every nonnegative integer $n$ we have
		\begin{align}\label{cco}
			c\left(5^\alpha \ell^2n+	y_{\alpha,\ell}\right)&
			\equiv\left\{\left(\frac{15}{\ell}\right)\left(1+\ell-\left(\frac{-24n-23}{\ell}\right)\right)\right\} 	c\left(5^\alpha n+\tilde{	\delta}_\alpha\right)	\nonumber\\&\quad-\ell 	c\left(\frac{5^\alpha n}{\ell^2}+	y_{\alpha,\ell^{-1}}\right)\quad\quad\quad\qquad\quad\qquad\pmod{5^{\alpha+1}}. 
		\end{align}
		\item
		If $\alpha\geq2$ is even, then for every nonnegative integer $n$ we have
		\begin{align}\label{cce}
			c\left(5^\alpha \ell^2n+	y_{\alpha,\ell}\right)& \equiv\left\{\left(\frac{15}{\ell}\right)\left(1+\ell-\ell^{2}\left(\frac{-24n-19}{\ell}\right)\right)\right\} 	c\left(5^\alpha n+\tilde{	\delta}_\alpha\right)
			\nonumber\\ &\quad-\ell 	c\left(\frac{5^\alpha n}{\ell^2}+	y_{\alpha,\ell^{-1}}\right) \quad\quad\quad\quad\quad\quad\quad\quad\pmod{5^{\alpha+1}}.
		\end{align}
		Here above, $(\frac{\cdot}{\ell})$ is the Legendre symbol.
	\end{enumerate}
\end{theorem}
As applications, we have the following refinements of \eqref{thcon}.
\begin{corollary}\label{co1}
	The following hold:
	\begin{enumerate}[(1)]
		\item 
		Let $7\leq\ell\equiv 3\pmod{5}$ be prime. 	If $\alpha\geq1$ is odd (resp. even) and $0\leq r\leq \ell-1$ is an integer satisfying 
		$\left(\frac{-24r-23}{\ell}\right)=-1$ (resp. 	$\left(\frac{-24r-19}{\ell}\right)=1$), then for every nonnegative integer $n$ we have
		\begin{align}
			&NT\left(1,5,5^{\alpha+1}\ell^3n+5^{\alpha+1}\ell^2r+5	y_{\alpha,\ell}+4\right)\nonumber\\&\equiv NT\left(4,5,5^{\alpha+1}\ell^3n+5^{\alpha+1}\ell^2r+5	y_{\alpha,\ell}+4\right)\pmod{5^{\alpha+1}}.\label{coc2}
		\end{align}
		\item 
		Let $\ell\equiv 4\pmod{5}$ be prime and $0\leq r, s\leq \ell-1$ be integers.
		\begin{enumerate}[(i)]
			\item 
			If $\alpha\geq1$ is odd,
			$
			24r+23\equiv0\pmod{\ell}$ and $
			24s\ell+24r+23\not\equiv0\pmod{\ell^2}.
			$
			\item
			If $\alpha\geq2$ is even,
			$
			24r+19\equiv0\pmod{\ell}$ and $
			24s\ell+24r+19\not\equiv0\pmod{\ell^2}.
			$
			
			Then for every nonnegative integer $n$ we have
			\begin{align}
				&NT\left(1,5,5^{\alpha+1}\ell^4n+5^{\alpha+1}\ell^3s+5^{\alpha+1}\ell^2r+5	y_{\alpha,\ell}+4\right)\nonumber\\&\equiv NT\left(4,5,5^{\alpha+1}\ell^4n+5^{\alpha+1}\ell^3s+5^{\alpha+1}\ell^2r+5	y_{\alpha,\ell}+4\right)\pmod{5^{\alpha+1}}.\label{coc1}
			\end{align}	
		\end{enumerate}
	\end{enumerate}
\end{corollary}
	
	The paper is organized as follows. We give some fundamental lemmas from \cite{HuntHirschhorn1981asimple} and prove Theorem \ref{thid} in Section \ref{s2}. The proof of Theorem \ref{th1} is provided in Section \ref{s3}. In Section \ref{s4}, we prove Theorem \ref{th2} and Corollary \ref{co1}.
	
\section{Proofs of Theorem \ref{thid}}\label{s2}
	Let $z:=z(q):=\left( q R\left(q^{5}\right)\right)^{-1}$ with
	$$
	R(q):=\prod_{n=1}^{\infty} \frac{\left(1-q^{5 n-4}\right)\left(1-q^{5 n-1}\right)}{\left(1-q^{5 n-3}\right)\left(1-q^{5 n-2}\right)}.
	$$
	
	\begin{lemma}
		We have
	\begin{align}
		z-1-\frac{1}{z}&= \frac{E_{1}}{qE_{25}}, \label{z1}\\
	u:=u(q):=	z^{5}-11-\frac{1}{z^{5}}&= \frac{E_{5}^{6}}{q^{5}E_{25}^{6}} ,\label{z2}\\
		H\left\{\left(\frac{z^{5}-11-z^{-5}}{z-1-z^{-1}}\right)^{i}\right\}&=\sum_{j \geq 1} m_{i, j} u^{i-j},\label{hi}\\
		\frac{	z^{5}-11-\frac{1}{z^{5}}}{	z-1-\frac{1}{z}}&=\frac{E_5^{6}}{q^{5} E_{25}^{6}} \bigg / \frac{E_1}{q E_{25}},\label{z3}
	\end{align}
%	where $u:=u(q):=z^{5}-11-z^{-5}$, and the 
where $m_{i, j}$ are defined by \eqref{defm1} and \eqref{defm2}. The operator $H$ picks out those terms in which the power is divided by 5.
	\end{lemma}
	Equations \eqref{z1}-\eqref{z2} were first proved by Watson \cite{watson1938ramanujans}, and \eqref{hi} was due to Hirschhorn and Hunt \cite{HuntHirschhorn1981asimple}. Equation \eqref{z3} follows from \eqref{z1} and \eqref{z2} immediately.
	%\cite[Lemma (2.7)]{HuntHirschhorn1981asimple}.

	\begin{lemma}
		Recall that $\tilde{a}_{i, j}, \tilde{b}_{i, j}$ are defined in \eqref{defab}.
	For $i\geq 1$, we have
		\begin{align}
			H\left\{\left(\frac{z^{5}-11-z^{-5}}{z-1-z^{-1}}\right)^{6i-4}\right\}=\sum_{j \geq 1} \tilde{a}_{i, j} u^{5i-3-j}\label{h6i40}
			\intertext{and}
			H\left\{\left(\frac{z^{5}-11-z^{-5}}{z-1-z^{-1}}\right)^{6i-5}\right\}=\sum_{j \geq 1} \tilde{b}_{i, j} u^{5i-4-j}\label{h6i50}.
		\end{align}
	\end{lemma}
	\begin{proof}
%	\cite[Lemma (2.7)]{HuntHirschhorn1981asimple} implies that, for $i\geq0$, $H\left\{\left(\frac{z^{5}-11-z^{-5}}{z-1-z^{-1}}\right)^{i}\right\}$ is a polynomial in $u$.
		Note that
	$$\frac{z^{5}-11-z^{-5}}{z-1-z^{-1}}=z^{4}+z^{3}+2 z^{2}+3 z+5-\frac{3}{z}+\frac{2}{z^{2}}-\frac{1}{z^{3}}+\frac{1}{z^{4}}.$$	
	This together with \eqref{hi} implies that $H\left\{\left(\frac{z^{5}-11-z^{-5}}{z-1-z^{-1}}\right)^{6i-4}\right\}$ is a polynomial in $u$ with degree less than 
	$5i-3.$ Assume that
		\begin{align}\label{h6i41}
		H\left\{\left(\frac{z^{5}-11-z^{-5}}{z-1-z^{-1}}\right)^{6i-4}\right\}=\sum_{j \geq 1} a^\prime_{i, j} u^{5i-3-j}.
	\end{align}
	Next, we replace $i$ by $6i-4$ in \eqref{hi} to obtain
	\begin{equation}
		H\left\{\left(\frac{z^{5}-11-z^{-5}}{z-1-z^{-1}}\right)^{6i-4}\right\}=\sum_{j \geq 1} m_{6i-4, j} u^{6i-4-j}.\no
	\end{equation}
	This together with \eqref{h6i41} gives
	$$a^\prime_{i, j}=m_{6i-4, i-1+j}=\tilde{a}_{i, j},$$ where the second equality follows from \eqref{defab}. Then \eqref{h6i40} follows.
	
	With a similar argument, one can prove \eqref{h6i50} and we omit the details.
		\end{proof}
	\begin{proof}[Proof of Theorem \ref{thid}]
		We proceed by induction on $\alpha.$
		
		Rewrite \eqref{gencn} as follows:
		\begin{align*}
			\sum_{n=0}^\infty c(n)q^n&=\frac{q^{4} E_{25}^{5}}{E_5^{2}}\left(\frac{E_5^{6}}{q^{5} E_{25}^{6}} \bigg / \frac{E_1}{q E_{25}}\right)
			\\&=\frac{q^{4} E_{25}^{5}}{E_5^{2}}\left(\frac{z^{5}-11-z^{-5}}{z-1-z^{-1}}\right).
		\end{align*}
	Extracting terms with $q^n$ where $n\equiv 4\pmod{5}$ yields
	\begin{align*}
		\sum_{n=0}^\infty c(5n+4)q^{5n+4}&=\frac{q^{4} E_{25}^{5}}{E_5^{2}}\cdot H\left(\frac{z^{5}-11-z^{-5}}{z-1-z^{-1}}\right)
	\\&=\frac{5q^{4} E_{25}^{5}}{E_5^{2}},
	\end{align*} 
	where the second equality follows from \eqref{hi}. Thus
	\begin{align*}
		\sum_{n=0}^\infty c(5n+4)q^{n}=\frac{5 E_{5}^{5}}{E_1^{2}},
	\end{align*} 
	which is the case $\alpha=1$ in \eqref{thidodd}.
	
	For $\alpha$ odd , we suppose that 
	\begin{align}
		\sum_{n=0}^\infty c\left(5^\alpha n+\frac{23\cdot 5^{\alpha}-19}{24}\right)q^n=		&\sum_{i \geq 1} \tilde{x}_{\alpha, i} q^{i-1} \frac{E_5^{6 i-1}}{E_1^{6 i-4}}.\no
	\end{align}
	It follows 
	\begin{align}
		\sum_{n=0}^\infty c\left(5^\alpha n+\frac{23\cdot 5^{\alpha}-19}{24}\right)q^n=		&\sum_{i \geq 1} \tilde{x}_{\alpha, i} \frac{q^{25i-17} E_{25}^{30i-20}}{E_5^{30i-23}}\left(\frac{E_5^{6}}{q^{5} E_{25}^{6}} \bigg / \frac{E_1}{q E_{25}}\right)^{6i-4}.\no
	\end{align}
	Extract terms with $q^n$ where $n\equiv 3\pmod{5}$ to obtain
	\begin{align}
	&	\sum_{n=0}^\infty c\left(5^{\alpha+1} n+\frac{19\cdot \left(5^{\alpha+1}-1\right)}{24}\right)q^{5n+3}\no\\&=		\sum_{i \geq 1} \tilde{x}_{\alpha, i} \frac{q^{25i-17} E_{25}^{30i-20}}{E_5^{30i-23}}H\left\{\left(\frac{E_5^{6}}{q^{5} E_{25}^{6}} \bigg / \frac{E_1}{q E_{25}}\right)^{6i-4}\right\}\no
	\\&=		\sum_{i \geq 1} \tilde{x}_{\alpha, i} \frac{q^{25i-17} E_{25}^{30i-20}}{E_5^{30i-23}}\cdot\left(\sum_{j \geq 1} \tilde{a}_{i, j} u^{5i-3-j}\right),\no
	\end{align}
	where the second equality follows from \eqref{h6i40}. Thus
	\begin{align}
		&	\sum_{n=0}^\infty c\left(5^{\alpha+1} n+\frac{19\cdot \left(5^{\alpha+1}-1\right)}{24}\right)q^{n}\no
		\\&=	\sum_{j \geq 1}\left(	\sum_{i \geq 1}\tilde{x}_{\alpha, i} \tilde{a}_{i, j}\frac{q^{5i-4} E_{5}^{30i-20}}{E_1^{30i-23}}\cdot \left[u\left(q^{\frac{1}{5}} \right)\right]^{5i-3-j}\right)\no
		\\&=	\sum_{j \geq 1} \tilde{x}_{\alpha+1, j} q^{j-1}\frac{ E_{5}^{6j-2}}{E_1^{6j-5}},\label{podd0}
	\end{align}
where the last equality follows from \eqref{def_x} and \eqref{z2}.

	Invoke \eqref{z3} again to rewrite \eqref{podd0} as follows:
	\begin{align}
		&	\sum_{n=0}^\infty c\left(5^{\alpha+1} n+\frac{19\cdot \left(5^{\alpha+1}-1\right)}{24}\right)q^{n}\no
		\\&=	\sum_{i \geq 1} \tilde{x}_{\alpha+1, i} \frac{q^{25i-21} E_{25}^{30i-25}}{E_5^{30i-28}}\left(\frac{E_5^{6}}{q^{5} E_{25}^{6}} \bigg / \frac{E_1}{q E_{25}}\right)^{6i-5}.\no
	\end{align}
	Extract terms with $q^n$ where $n\equiv 4\pmod{5}$,
	\begin{align}
		&	\sum_{n=0}^\infty c\left(5^{\alpha+2} n+\frac{115\cdot 5^{\alpha+1}-19}{24}\right)q^{5n+4}\no \\
		&=	\sum_{i \geq 1} \tilde{x}_{\alpha+1, i} \frac{q^{25i-21} E_{25}^{30i-25}}{E_5^{30i-28}}H\left\{\left(\frac{E_5^{6}}{q^{5} E_{25}^{6}} \bigg / \frac{E_1}{q E_{25}}\right)^{6i-5}\right\}\no \\
		&=	\sum_{i \geq 1} \tilde{x}_{\alpha+1, i} \frac{q^{25i-21} E_{25}^{30i-25}}{E_5^{30i-28}}\cdot\left(\sum_{j \geq 1}\tilde{b}_{i, j} u^{5i-4-j}\right),\no
	\end{align}
	where the second equality follows from \eqref{h6i50}. Then
	\begin{align}
		&	\sum_{n=0}^\infty c\left(5^{\alpha+2} n+\frac{115\cdot 5^{\alpha+1}-19}{24}\right)q^{n}\no
		%\\&=	\sum_{j \geq 1} \tilde{x}_{\alpha+1, j} \frac{q^{25j-21} E_{25}^{30j-25}}{E_5^{30j-28}}H\left(\frac{E_5^{6}}{q^{5} E_{25}^{6}} \bigg / \frac{E_1}{q E_{25}}\right)^{6j-5}\no
		\\&=	\sum_{j \geq 1}\left( \sum_{i \geq 1} \tilde{x}_{\alpha+1, i} \tilde{b}_{i, j} \frac{q^{5i-5} E_{5}^{30i-25}}{E_1^{30i-28}}\cdot \left[u\left(q^{\frac{1}{5}} \right)\right]^{5i-4-j}\right)\no\\
		&=	\sum_{j \geq 1}\tilde{x}_{\alpha+2, j}q^{j-1}\frac{ E_{5}^{6j-1}}{E_1^{6j-4}}.\no
	\end{align}
	So the proof of \eqref{thidodd} is concluded by induction. Then \eqref{thideven} follows from \eqref{podd0}. This completes the proof of Theorem \ref{thid}.
	\end{proof}
	
\section{Proofs of Theorem \ref{th1}}\label{s3}
	\begin{lemma}
		We have
		\begin{align}
			\nu\left(\tilde{a}_{i, j}\right) \geq\left[\frac{1}{2}(5 j-i-2)\right], \quad \nu\left(\tilde{b}_{i, j}\right) \geq\left[\frac{1}{2}(5 j-i-1)\right],
		\end{align}
		where $\nu(n)$ denotes the exact power of $5$ dividing $n$.
	\end{lemma}
\begin{proof}	
Recall \cite[Lemma 4.1]{HuntHirschhorn1981asimple} that:
\begin{align}
	\nu\left(m_{i, j}\right) \geq\left[\frac{1}{2}(5 j-i-1)\right].\no
\end{align}
Thus 
$$
\begin{aligned}
	\nu\left(\tilde{a}_{i, j}\right) 
	& =\nu\left(m_{6 i-4, i-1+j}\right) \geq\left[\frac{1}{2}(5(i-1+j)-(6 i-4)-1)\right]=\left[\frac{1}{2}(5 j-i-2)\right], \\
	\nu\left(\tilde{b}_{i, j}\right)
	& =\nu\left(m_{6 i-5, i-1+j}\right) \geq\left[\frac{1}{2}(5(i-1+j)-(6 i-5)-1)\right]=\left[\frac{1}{2}(5 j-i-1)\right].\qedhere
\end{aligned}
$$
\end{proof}	
\begin{lemma}\label{lex}
	For $\alpha, i\geq 1,$ we have 

	\begin{align}
		\nu\left(\tilde{x}_{\alpha, i}\right) &\geq \alpha+\left[\frac{1}{2}(5 i-5)\right].\label{xodd}
	\end{align}
\end{lemma}
\begin{proof}
We prove by induction on $\alpha$.

	When $\alpha=1$, \eqref{xodd} is obvious since $\tilde{\bf x}_1=(5,0,0,\cdots)$.
	Suppose $\alpha$ is odd, and
\begin{align}
	\nu\left(\tilde{x}_{\alpha, i}\right) &\geq \alpha+\left[\frac{1}{2}(5 i-5)\right].\no
\end{align}
Then
\begin{align*}
	\nu\left(\tilde{x}_{\alpha+1, i}\right)&=	\nu\left(\sum_{j\geq1}\tilde{x}_{\alpha, j}\tilde{a}_{j,i}\right) \\
	&\geq\min _{j \geq 1}\left\{\nu\left(\tilde{x}_{\alpha, j}\right)+\nu\left(\tilde{a}_{j, i}\right)\right\}
	\\
	&\geq\min_{j\geq1}\left\{\alpha+\left[\frac{1}{2}(5 j-5)\right]+\left[\frac{1}{2}(5 i-j-2)\right]\right\}\\
	&=(\alpha+1)+\left[\frac{1}{2}(5 i-5)\right].
\end{align*}
Next, suppose $\alpha$ is even, and 
\begin{align*}
	\nu\left(\tilde{x}_{\alpha, i}\right)\geq\alpha+\left[\frac{1}{2}(5 i-5)\right].
\end{align*}
Then
\begin{align*}
	\nu\left(\tilde{x}_{\alpha+1, i}\right)&=	\nu\left(\sum_{j\geq 1}\tilde{x}_{\alpha, j}\tilde{b}_{j,i}\right) \\
	& \geq \min _{j \geq 1}\left\{\nu\left(\tilde{x}_{\alpha, j}\right)+\nu\left(\tilde{b}_{j, i}\right)\right\} \\
	&\geq \min _{j \geq 1}\left\{\alpha+\left[\frac{1}{2}(5 j-5)\right]+\left[\frac{1}{2}(5 i-j-1)\right]\right\}\\
	&=(\alpha+1)+\left[\frac{1}{2}(5 i-4)\right],
%	& > (\alpha+1)+\left[\frac{1}{2}(5 i-5)\right], 
\end{align*}
which completes the proof of Lemma \ref{lex}.
	\end{proof}
%We are now in a position to prove Theorem \ref{th1}.
\begin{proof}[Proof of Theorem \ref{th1}]
	Lemma \ref{lex} implies that $\nu\left(\tilde{x}_{\alpha, i}\right)\geq \alpha,$ for all $\alpha, i\geq1$. Thus $\tilde{x}_{\alpha, i}\equiv0\pmod{5^\alpha}$.
Then we deduce from Theorem \ref{thid} that 
\begin{align}
	c\left(5^\alpha n+	\tilde{	\delta}_\alpha\right) \equiv0\pmod{5^\alpha},\label{prf_thm1.1_1}
\end{align}
 which together with \eqref{gencn} gives Theorem \ref{th1}.
	\end{proof}

\section{Proofs of Theorem \ref{th2} and Corollary \ref{co1}}\label{s4}

We need the following lemma.
\begin{lemma}\label{leab}
Define
\begin{align*}
	&F(q):=q^{19}E_{24}^{19}=:\sum_{n=1}^{\infty} a(n) q^{n}, \\
	&G(q):=q^{23}E_{24}^{23}=:\sum_{n=1}^{\infty} b(n) q^{n}.
\end{align*}	
Then we have 
\begin{align}
	c\left(5^\alpha n+\tilde{	\delta}_\alpha\right) &\equiv \tilde{x}_{\alpha, 1}  b(24n+23)  \pmod{5^{\alpha+1}},  & \alpha \textrm{ odd, }\label{cal1} \\
		c\left(5^\alpha n+\tilde{	\delta}_\alpha\right) &\equiv \tilde{x}_{\alpha, 1}  a(24n+19)  \pmod{5^{\alpha+1}}. & \alpha \textrm{ even.} \label{cal2}
\end{align}
\end{lemma}
\begin{proof}
	Noting that $5^{\alpha+1}\mid \tilde{x}_{\alpha, i}$ for $i\geq2$, we deduce from \eqref{thidodd} that, if $\alpha$ is odd, then 
	\begin{align*}
		&	\sum_{n=0}^\infty c\left(5^\alpha n+\tilde{	\delta}_\alpha\right)q^{24n+23}\\ 
		& \equiv	\tilde{x}_{\alpha, 1} \frac{q^{23}E_{120}^{5}}{E_{24}^{2}}\\
		& \equiv	\tilde{x}_{\alpha, 1}  G(q)\\
		& \equiv	\tilde{x}_{\alpha, 1}\sum_{n=1}^{\infty} b(n) q^{n} \pmod{5^{\alpha+1}},
	\end{align*}
which gives \eqref{cal1}.

	Similarly, we can apply \eqref{thideven} to obtain \eqref{cal2}.
	\end{proof}
We are now in a position to prove Theorem \ref{th2}.
\begin{proof}[Proof of Theorem \ref{th2}]
	%Replacing $n$ by $24n+23$ (resp. $24n+19$) in \eqref{conb} (resp. \eqref{cona}), multiplying by $x_{\alpha, 1}$ on both sides of the resulting equation and invoking Lemma \ref{leab}, we obtain 
	%\eqref{cco} (resp. \eqref{cce}). 
From \cite[(12) and (13)]{love} we know that, for primes $\ell\geq5$, 
\begin{align}
	&b\left(\ell^{2} n\right) \equiv\left\{\left(\frac{15}{\ell}\right)\left(1+\ell-\left(\frac{-n}{\ell}\right)\right)\right\} b(n)-\ell b\left(n / \ell^{2}\right) & \pmod{5},\label{conb}\\
		&a\left(\ell^{2} n\right) \equiv\left\{\left(\frac{15}{\ell}\right)\left(1+\ell-\ell^{2}\left(\frac{-n}{\ell}\right)\right)\right\} a(n)-\ell a\left(n / \ell^{2}\right) & \pmod{5}. \label{cona}
\end{align}
Replacing $n$ by $24n+23$ in \eqref{conb} gives
\begin{align*}
	b\left(\ell^2 (24n+23)\right)
	\equiv & \left\{ \left(\frac{15}{\ell}\right)\left(1+\ell-\left(\frac{-24n-23}{\ell}\right)\right)\right\} b(24n+23) \\
	& -\ell b\left( \frac{24n+23}{\ell^2}\right) \pmod{5}.
\end{align*}
Multiplying by $\tilde{x}_{\alpha, 1}$ on both sides of the above congruence and invoking \eqref{cal1} lead to
\begin{align}
	c\left(5^\alpha \frac{l^2(24n+23)-23}{24}+ \tilde{\delta}_\alpha\right) 
	\equiv & \left\{ \left(\frac{15}{\ell}\right)\left(1+\ell-\left(\frac{-24n-23}{\ell}\right)\right)\right\} c(5^\alpha n+ \tilde{\delta}_\alpha) \notag \\
	& - \ell c\left(5^\alpha \frac{\frac{24n+23}{\ell ^2}-23}{24}
	+ \tilde{\delta}_\alpha\right) \pmod{5^{\alpha+1}}, \no
\end{align}
which gives \eqref{cco}. Similarly, replacing $n$ by $24n+19$ in \eqref{cona}, multiplying by $\tilde{x}_{\alpha, 1}$ on both sides of the resulting equation and invoking \eqref{cal2}, we obtain \eqref{cce}. 
\end{proof}

\begin{proof}[Proof of Corollary \ref{co1}]
		(1) Let $7\le \ell \equiv 3\pmod 5$ be primes. For odd $\alpha \ge 1$ and $0\le r\le \ell -1$ satisfying $\left(\frac{-24r-23}{\ell}\right)=-1$,  replacing $n$ by $\ell n+r$ in \eqref{cco} gives
		\begin{align}
			& c\left(5^\alpha \ell^2(\ell n+r)+	y_{\alpha,\ell}\right) \notag \\
			&\equiv  \left\{\left(\frac{15}{\ell}\right)\left(1+\ell-\left(\frac{-24(\ell n+r)-23}{\ell}\right)\right)\right\} 	c\left(5^\alpha (\ell n+r)+\tilde{	\delta}_\alpha\right)	\notag \\
			& \quad-\ell 	c\left(\frac{5^\alpha (\ell n+r)}{\ell^2}+	y_{\alpha,\ell^{-1}}\right)  \hspace{4cm}\pmod{5^{\alpha+1}} .%\label{prf_cor1.4_1}.
		\end{align}
Since $\frac{5^\alpha (\ell n+r)}{\ell^2}+	y_{\alpha,\ell^{-1}}$ is not a integer when $\left(\frac{-24r-23}{\ell}\right)=-1$, we have $$	c\left(\frac{5^\alpha (\ell n+r)}{\ell^2}+	y_{\alpha,\ell^{-1}}\right) =0.$$Thus
	$
		 c\left(5^\alpha \ell^2(\ell n+r)+	y_{\alpha,\ell}\right) 
%		\equiv & \left\{\left(\frac{15}{\ell}\right)\left(1+\ell-\left(\frac{-24(\ell n+r)-23}{\ell}\right)\right)\right\} 	c\left(5^\alpha (\ell n+r)+\tilde{	\delta}_\alpha\right)	\notag \\
%		& -\ell 	c\left(\frac{5^\alpha (\ell n+r)}{\ell^2}+	y_{\alpha,\ell^{-1}}\right) \notag \\
		\equiv  \left(\frac{15}{\ell}\right)(\ell+2)
		c \left(5^\alpha (\ell n+r)+\tilde{	\delta}_\alpha\right)
		\equiv  ~0 \pmod{5^{\alpha+1}},%\label{prf_cor1.4_1}.
$ where the last congruence follows from \eqref{prf_thm1.1_1}.
This together with \eqref{gencn} gives
		 \eqref{coc2} with $\alpha$ odd. Similarly, we can apply \eqref{cce} to prove \eqref{coc2} with $\alpha$ even.
		
		(2) Let $\ell\equiv 4\pmod{5}$ be primes and $0\leq r, s\leq \ell-1$ be integers satisfying $
		24r+23\equiv0\pmod{\ell}$ and $24s\ell+24r+23\not\equiv 0 \pmod{\ell^2}$. When $\alpha$ is odd, replacing $n$ by $\ell^2n+\ell s+r$  in \eqref{cco} and noting that $\ell\equiv 4\pmod{5},
		24r+23\equiv0\pmod{\ell}$,
		we find that
		\begin{align}
			& c\left(5^\alpha \ell^2(\ell^2n+\ell s+r)+y_{\alpha,\ell}\right) \notag \\
			&\equiv  -\ell 	c\left(\frac{5^\alpha (\ell^2 n+\ell s+r)}{\ell^2}+	y_{\alpha,\ell^{-1}}\right)
		\pmod{5^{\alpha+1}} .\no
		\end{align}
	We have $c\left(\frac{5^\alpha (\ell^2 n+\ell s+r)}{\ell^2}+	y_{\alpha,\ell^{-1}}\right)
=0$ since $\frac{5^\alpha (\ell^2 n+\ell s+r)}{\ell^2}+	y_{\alpha,\ell^{-1}}$ is not a integer when
	$24s\ell+24r+23\not\equiv 0 \pmod{\ell^2}.$ Thus
		\begin{align}
		& c\left(5^\alpha \ell^2(\ell^2n+\ell s+r)+y_{\alpha,\ell}\right) 
		\equiv 0 \pmod{5^{\alpha+1}} ,\no
	\end{align}
	which together with \eqref{gencn}  gives \eqref{coc1} with $\alpha$ odd. With a similar argument, we can prove
	\eqref{coc1} with $\alpha$ even after replacing $n$ by $\ell^2n+\ell s+r$  in \eqref{cce}.
\end{proof}
%\begin{align}
%		c\left(5^\alpha \ell^2n+	y_{\alpha,\ell}\right)&
% \equiv\left\{\left(\frac{15}{\ell}\right)\left(1+\ell-\left(\frac{-24n-23}{\ell}\right)\right)\right\} 	c\left(5^\alpha n+\tilde{	\delta}_\alpha\right)	\nonumber\\&\quad-\ell 	c\left(\frac{5^\alpha n}{\ell^2}+	y_{\alpha,\ell^{-1}}\right)\quad\quad\quad\quad\quad\quad\pmod{5^{\alpha+1}}  \textrm{ if $\alpha$ is odd } \intertext{and}
%		c\left(5^\alpha \ell^2n+	y_{\alpha,\ell}\right)& \equiv\left\{\left(\frac{15}{\ell}\right)\left(1+\ell-\ell^{2}\left(\frac{-24n-19}{\ell}\right)\right)\right\} 	c\left(5^\alpha n+\tilde{	\delta}_\alpha\right)
%		\nonumber\\ &\quad-\ell 	c\left(\frac{5^\alpha n}{\ell^2}+	y_{\alpha,\ell^{-1}}\right) \quad\quad\quad\quad\quad\quad\pmod{5^{\alpha+1}}  \textrm{ if $\alpha$ is even.} 
%	\end{align}

\section*{Acknowledgements}
	This work is supported by NSFC (National Natural Science Foundation of China) through Grants NSFC 12071331 and NSFC 11971341.

\end{document}